\documentclass[11pt]{article}

\usepackage{amsmath, amsfonts, amssymb, amsthm, mathrsfs}
\usepackage{mathtools,bbm}
\usepackage{enumerate}
\usepackage{graphicx}
\usepackage{color}
\usepackage[margin=1in]{geometry}
\usepackage{lineno}
\usepackage[ruled, vlined, algonl]{algorithm2e}
\usepackage{float}
\usepackage{bm}
\usepackage{todonotes}

\newcommand{\argmin}{\operatornamewithlimits{argmin}}

\DeclareMathOperator{\prox}{prox}

\DeclareMathOperator{\sign}{sgn}

\newcommand{\x}[1]{x^{(#1)}}

\newtheorem{theorem}{Theorem}

\newtheorem{lemma}{Lemma}

\newtheorem{proposition}{Proposition}

\theoremstyle{remark}

\begin{document}

\title{The Proximity Operator of The Log-sum Penalty}
\author{Ashley Prater-Bennette\thanks{Air Force Research Laboratory, Rome, NY 13441. Email: \texttt{ashley.prater-bennette@us.af.mil}}  \and Lixin Shen\thanks{Department of Mathematics, Syracuse University, Syracuse, NY 13244. Email: \texttt{lshen03@syr.edu}}
\and Erin E. Tripp\thanks{Air Force Research Laboratory, Rome, NY 13441. Email: \texttt{erin.tripp.4@us.af.mil}}}
\maketitle

\begin{abstract}
The log-sum penalty is often adopted as a replacement for the $\ell_0$ pseudo-norm in compressive sensing and low-rank optimization. The hard-thresholding operator, i.e., the proximity operator of the $\ell_0$ penalty, plays an essential role in applications; similarly, we require an efficient method for evaluating the proximity operator of the log-sum penalty. Due to the nonconvexity of this function, its proximity operator is commonly computed through the iteratively reweighted $\ell_1$ method, which replaces the log-sum term with its first-order approximation. This paper reports that the proximity operator of the log-sum penalty actually has an explicit expression. With it, we show that the iteratively reweighted $\ell_1$ solution disagrees with the true proximity operator of the log-sum penalty in certain regions. As a by-product, the iteratively reweighted $\ell_1$ solution is precisely characterized in terms of the chosen initialization. We also give the explicit form of the proximity operator for the composition of the log-sum penalty with the singular value function, as seen in low-rank applications. These results should be useful in the development of efficient and accurate algorithms for optimization problems involving the log-sum penalty.
\end{abstract}

\section{Introduction}
The log-sum penalty function is defined as
\begin{equation}\label{eq:def:logf}
f(x) := \sum_{i=1}^n \log\left(1+\frac{|x_i|}{\epsilon}\right),
\end{equation}
where $\epsilon>0$ and $x=(x_1,x_2,\ldots,x_n)\in \mathbb{R}^n$. This function is commonly used to bridge the gap between the $\ell_0$ and $\ell_1$ norms in compressive sensing \cite{Candes-Wakin-Boyd:JFAA:08,Shen-Fang-Li:LogSum:2013} and as a nonconvex surrogate function of the matrix rank function in the low-rank regularization \cite{cai2020image,Deng-Dai-Liu-Zhang-Hu:IEEENNLS:13,Dong-Shi-Li-Ma-Huang:IEEEIP:2014,Fazel-Hindi-Boyd:03}.

When adopted in a compressive sensing problem, it essentially requires solving the following subproblem
\begin{equation}\label{eq:prox-f-CS}\tag{P\textsubscript{1}}
\min \left\{\frac{1}{2\lambda} \|x-z\|^2+\sum_{i=1}^n \log\left(1+\frac{|x_i|}{\epsilon}\right): x\in\mathbb{R}^n\right\},
\end{equation}
where $\lambda>0$ is a regularization parameter. In the language of convex analysis, the solution to \eqref{eq:prox-f-CS} is precisely the proximity operator of $f$ with index $\lambda$ at $z$ (see, e.g., \cite{Bauschke-Combettes:11}). Due to the nonconvexity of the objective, this problem is difficult to solve directly. Instead, an approximate solution is typically obtained through the iteratively reweighted $\ell_1$ minimization method, which sequentially linearizes $f$ around the current iterate and solves the linearized convex problem to obtain the next iterate \cite{Candes-Wakin-Boyd:JFAA:08}.

Similarly, an essential step in algorithms for low-rank optimization problems is  solving
\begin{equation}\label{eq:prox-f-Low-Rank}\tag{P\textsubscript{2}}
\min \left\{\frac{1}{2\lambda} \|X-Z\|^2_F+\sum_{i=1}^{m \land n} \log\left(1+\frac{\sigma_i(X)}{\epsilon}\right): X\in\mathbb{R}^{m \times n}\right\},
\end{equation}
where $\sigma_i(X)$ is the $i$th singular value of $X$. Here, ${m \land n}:=\min\{m,n\}$ and $\|\cdot\|_F$ denotes the Frobenius norm. A similar strategy in solving~\eqref{eq:prox-f-CS} is applied for solving \eqref{eq:prox-f-Low-Rank}.  The solution to \eqref{eq:prox-f-Low-Rank} is the proximity operator of $f\circ \sigma$ with index $\lambda$ at $Z$, where $\sigma: \mathbb{R}^{m \times n} \rightarrow \mathbb{R}^{m \land n}$ gives all singular values of a matrix. The regularization term in \eqref{eq:prox-f-Low-Rank} is the Log-Det heuristic used in \cite{Fazel-Hindi-Boyd:03}. In fact, if $m \le n$, then
$$
\sum_{i=1}^{m \land n} \log\left(1+\frac{\sigma_i(X)}{\epsilon}\right) = \log \mathrm{det} \left(I+\frac{1}{\epsilon} (X X^\top)^{1/2}\right).
$$
If $m >n$, we simply replace $X X^\top$ by $X^\top X$ in the above equation.

Notice that the log-sum function is additively separable; that is,
\begin{equation*}\label{eq:f-g}
f(x)= \sum_{i=1}^n g(x_i),
\end{equation*}
where $g: \mathbb{R} \rightarrow \mathbb{R}$ is defined by
\begin{equation}\label{eq:def:logg}
g(w) = \log\left(1+\frac{|w|}{\epsilon}\right).
\end{equation}
As a result, the solutions to \eqref{eq:prox-f-CS} and \eqref{eq:prox-f-Low-Rank} can be given in terms of the proximity operator of $g$. Throughout this paper, $f$ and $g$ always refer to the functions given in \eqref{eq:def:logf} and \eqref{eq:def:logg}, respectively.

The purpose of this paper is to show that there exist closed-form solutions to \eqref{eq:prox-f-CS} and \eqref{eq:prox-f-Low-Rank}, and, therefore, time-consuming iterative procedures can be avoided. These expressions do not appear in the existing literature to the best of our knowledge and should improve the efficiency and accuracy of algorithms in compressive sensing and low-rank minimization where the function $f$ is used. While finalizing this work, we became aware of a recent paper \cite{Xia-Wang-Meng-Chai-Liang:JCM-2017} which attempts to find the proximity operator of $f$ under the condition $\sqrt{\lambda} >\epsilon$. However, the results presented there are innacurate.

We remark that since the objective functions in \eqref{eq:prox-f-CS} and \eqref{eq:prox-f-Low-Rank} are nonconvex, the sequences generated by the iterative scheme described above may not converge to a global solution of the corresponding optimization problem. In fact, we identify under what circumstances the iteratively reweighted algorithm for problem \eqref{eq:prox-f-CS} does not produce an optimal solution.

The rest of the paper is outlined as follows: In the next section, we give an explicit expression of the proximity operator of $g$, followed by an explicit expression of the solution to \eqref{eq:prox-f-CS}. With this, we show in Section~\ref{sec:rwl1} that the iteratively reweighted $\ell_1$ solution to \eqref{eq:prox-f-CS} disagrees with the true proximity operator of the log-sum penalty in certain regions. These regions are completely determined by the chosen initial guess for the reweighted $\ell_1$ algorithm. In Section~\ref{sec:Low-Rank},  we give an explicit expression of solutions to \eqref{eq:prox-f-Low-Rank}. Our conclusions are drawn in Section~\ref{sec:conclusions}.

\section{Solutions to Optimization Problem~\eqref{eq:prox-f-CS}}\label{sec:SSPF}
We begin in \ref{subsub:0} by collecting some fundamental lemmas related to the proximity operator of $g$. In Subsection~\ref{subsub:1}, we give the explicit expression of the proximity operator of $g$ then use it to derive the proximity operator of $f$.

\subsection{Basic Properties}\label{subsub:0}

The proximity operator of $g$ at $z \in  \mathbb{R}$ with index $\lambda$ is defined by
$$
\mathrm{prox}_{\lambda g} (z) := \argmin \left\{\frac{1}{2\lambda} \|w-z\|^2+g(w): w\in\mathbb{R}\right\}.
$$
The proximity operator is a set-valued mapping from $\mathbb{R} \rightarrow 2^{\mathbb{R}}$, the power set of $\mathbb{R}$. Because $g$, as defined by \eqref{eq:def:logg}, is continuous and coercive, the set $\mathrm{prox}_{\lambda g}(z)$ is not empty for any $\lambda > 0$ and $z\in \mathbb{R}$. 

By definition, the elements of this set are solutions of an optimization problem, and in order to characterize these solutions, we must understand the behavior of the objective function around its critical points. For given $\lambda$ and $z$, define $q_{\lambda, z}: \mathbb{R} \rightarrow \mathbb{R}$ as follows
\begin{equation*}\label{def:q}
q_{\lambda, z}(x)= \frac{1}{2\lambda} (x-z)^2 + g(x).
\end{equation*}
Note that $q_{\lambda, z}$ is differentiable away from the origin with
\begin{equation}\label{eq:qderiv}
\frac{d}{dx}q_{\lambda, z}(x) = \frac{1}{\lambda}(x-z) + \frac{1}{x + \epsilon\sign(x)}.
\end{equation}
Clearly,
$$
\mathrm{prox}_{\lambda g}(z) = \mathrm{arg}\min \{q_{\lambda, z}(x): x\in \mathbb{R}\}.
$$
A straightforward consequence of this definition is that $\prox_{\lambda g}$ is symmetric about the origin and acts as a shrinkage operator.

\begin{lemma}\label{lemma:ProxzSmallerz}
Let $z$ be nonzero. Then
(i) $\mathrm{prox}_{\lambda g}(z)=-\mathrm{prox}_{\lambda g}(-z)$, and (ii) $\mathrm{prox}_{\lambda g}(z) \subseteq [0, z)$ if $z$ is positive and $\mathrm{prox}_{\lambda g}(z) \subseteq (z, 0]$ if $z$ is negative.
\end{lemma}
\begin{proof}\ \ (i) This follows directly from the fact that $q_{\lambda, z}(x) = q_{\lambda, -z}(-x)$ for all $x \in \mathbb{R}$.

(ii) First assume $z>0$. One can check that $q_{\lambda, z}(x)<q_{\lambda, z}(-x)$, for $x>0$.
Hence the elements in $\mathrm{prox}_{\lambda g}(z)$ should be nonnegative. It can be verified from \eqref{eq:qderiv} that $q_{\lambda, z}(x)$ as a function of $x$ is increasing on $[z, \infty)$, which implies that $\mathrm{prox}_{\lambda g}(z) \subseteq  [0, z]$. By using Taylor's expansion for expanding $q_{\lambda, z}(x)$ at $z$, one has
$$
q_{\lambda, z}(x)=q_{\lambda, z}(z)+\frac{1}{z+\epsilon}(x-z)+\left(\frac{1}{2\lambda}-\frac{1}{2(z+\epsilon)^2}\right)(x-z)^2 +o(|x-z|^2).
$$
From this expression, we see that $q_{\lambda, z}(x)<q_{\lambda, z}(z)$ when $x$ is close enough to $z$ from below. We conclude that $\mathrm{prox}_{\lambda g}(z) \subseteq  [0, z)$.

The above discussion, along with (i), implies that $\mathrm{prox}_{\lambda g}(z) \subseteq (z, 0]$ if $z$ is negative.
 \end{proof}

By item (i) of Lemma \ref{lemma:ProxzSmallerz}, it is sufficient to study the proximity operator $\mathrm{prox}_{\lambda g}(z)$ for all non-negative $z$. Moreover, it follows immediately that for all $\lambda>0$,
\begin{equation*}\label{eq:prox-g-0}
\mathrm{prox}_{\lambda g}(0)=\{0\}.
\end{equation*}
For the rest of this section, we consider only $z>0$. Along with item (ii), we therefore only need to investigate the behavior of $q_{\lambda, z}(x)$ for $x \geq 0$. In this case, the derivative of $q_{\lambda, z}$ can be rewritten as
\begin{equation}\label{eq:qderiv-1}
\frac{d }{dx}q_{\lambda, z}(x)=\frac{(x-\frac{1}{2}(z-\epsilon))^2+\lambda-\frac{1}{4}(z+\epsilon)^2}{\lambda(x+\epsilon)}.
\end{equation}
For $z \ge \max\{2\sqrt{\lambda}-\epsilon,0\}$, the expression above can be factored as
\begin{equation}\label{eq:qderiv-2}
\frac{d }{dx}q_{\lambda, z}(x)=\frac{1}{\lambda(x+\epsilon)}(x-r_1(z))(x-r_2(z)),
\end{equation}
where
\begin{equation} \label{def:r1}
r_1(z):=\frac{1}{2}(z-\epsilon) - \sqrt{\frac{1}{4}(z+\epsilon)^2 - \lambda}
\end{equation}
and
\begin{equation}\label{def:r2}
r_2(z):=\frac{1}{2}(z-\epsilon) + \sqrt{\frac{1}{4}(z+\epsilon)^2 - \lambda}.
\end{equation}

The behavior of $q_{\lambda, z}$ depends on the relationship between the parameters $\lambda$ and $\epsilon$ as well as the magnitude of $z$, as described in the following lemmas. 

\begin{lemma}\label{lemma:qproperties}
For given $\lambda>0$ and $\epsilon > 0$,  the following statements hold.
\begin{enumerate}[(i)]
\item If $\sqrt{\lambda}\le \epsilon$, then $q_{\lambda, z}$ is increasing on $[0, \infty)$ for $z \in [0, \frac{\lambda}{\epsilon}]$; $q_{\lambda, z}$ is decreasing on $[0, r_2(z)]$ and increasing on $[r_2(z), \infty)$ for $z \in [\frac{\lambda}{\epsilon}, \infty)$.
\item If $\sqrt{\lambda}> \epsilon$, then $q_{\lambda, z}$ is increasing on $[0, \infty)$ for $z \in [0,  2\sqrt{\lambda}-\epsilon]$; $q_{\lambda, z}$ is increasing on $[0, r_1(z)]$, decreasing on $[r_1(z),r_2(z)]$, and increasing on $[r_2(z), \infty)$ for $z \in [2\sqrt{\lambda}-\epsilon, \frac{\lambda}{\epsilon}]$; $q_{\lambda, z}$ is decreasing on $[0, r_2(z)]$ and increasing on $[r_2(z), \infty)$ for $z \in [\frac{\lambda}{\epsilon}, \infty)$.

\end{enumerate}
\end{lemma}
\begin{proof}\ \
First, a few remarks on the functions $r_1$ and $r_2$. Recall that these functions are defined only for $z \geq \max\{2\sqrt{\lambda} - \epsilon, 0\}$. On this domain, we may determine the sign of $\frac{d}{dx} q_{\lambda, z}(x)$ by determining the sign of its factors. Clearly $r_2(z) > r_1(z)$ for all $z$ in this domain, and $r_2(\frac{\lambda}{\epsilon}) = 0$, noting that $\frac{\lambda}{\epsilon}=(2\sqrt{\lambda}-\epsilon)+(\sqrt{\epsilon}-\sqrt{\lambda/\epsilon})^2>2\sqrt{\lambda}-\epsilon$. Moreover,
$$
r'_1(z) = \frac{\sqrt{(z+\epsilon)^2-4\lambda}-(z+\epsilon)}{2\sqrt{(z+\epsilon)^2-4\lambda}}<0 \quad \mbox{and} \quad
r'_2(z) = \frac{\sqrt{(z+\epsilon)^2-4\lambda}+(z+\epsilon)}{2\sqrt{(z+\epsilon)^2-4\lambda}}>0.
$$
That is, $r_1$ is strictly decreasing, while $r_2$ is strictly increasing. The following proof relies on these observations, with care being taken around boundary points and zeros.

(i) If $\sqrt{\lambda}\le \epsilon$, we consider two situations, namely $2\sqrt{\lambda} \le \epsilon$ and  $2\sqrt{\lambda} > \epsilon$.
For $2\sqrt{\lambda} \le \epsilon$, the expression in \eqref{eq:qderiv-2} holds for all $z \ge 0$. Since $r_1(0)<0$, then $r_1(z) <0$ for all $z \ge 0$. Since $r_2(0)<0$ and $r_2(\lambda/\epsilon)=0$, then $r_2(z)<0$ for $z\in (0, \lambda/\epsilon)$ and $r_2(z)>0$ for $z\in (\lambda/\epsilon, \infty)$.  Based on these observations and \eqref{eq:qderiv-2}, $\frac{d}{dx} q_{\lambda, z}(x) < 0$ only when $x \in [0, r_2(z))$ and $z \in (\frac{\lambda}{\epsilon}, \infty)$.

For $2\sqrt{\lambda} > \epsilon$, we have $\lambda-\frac{1}{4}(z+\epsilon)^2 \ge 0$. Hence from \eqref{eq:qderiv-1}, $\frac{d}{dx}q_{\lambda, z}(x)>0$ for all $x \in [0,\infty)$ when $z\in [0, 2\sqrt{\lambda} - \epsilon]$. In the rest of discussion, we consider $z  \in (2\sqrt{\lambda} - \epsilon, \infty)$, for which expression in \eqref{eq:qderiv-2} holds. We have
$r_1(2\sqrt{\lambda} - \epsilon)=r_2(2\sqrt{\lambda} - \epsilon)=\sqrt{\lambda}-\epsilon \le 0$ and $r_2(\lambda/\epsilon)=0$. Because $r_1$ is strictly decreasing and $r_2$ is strictly increasing, we conclude that  
\begin{itemize} 
	\item $\frac{d }{dx}q_{\lambda, z}(x)>0$ for all $x \in [0,\infty)$ when $z\in [2\sqrt{\lambda}-\epsilon, \lambda/\epsilon)$, and 
	\item $\frac{d }{dx}q_{\lambda, z}(x)<0$ for all $x \in [0,r_2(z))$ and $\frac{d }{dx}q_{\lambda, z}(x)>0$ for all $x \in (r_2(z), \infty)$ when $z\in (\lambda/\epsilon, \infty)$.
\end{itemize} 
Thus item (i) holds regardless of the relationship between $2\sqrt{\lambda}$ and $\epsilon$.

(ii) It is clear from the above discussion that $\frac{d}{dx}q_{\lambda, z}(x)>0$ for all $x \in [0,\infty)$ when $z\in [0, 2\sqrt{\lambda} - \epsilon]$. Notice that $r_1(2\sqrt{\lambda} - \epsilon)=r_2(2\sqrt{\lambda} - \epsilon)=\sqrt{\lambda}-\epsilon > 0$ and $r_1(\lambda/\epsilon)=0$. Since $r_1(z)$ is strictly decreasing and $r_2(z)$ is strictly increasing on $[2\sqrt{\lambda} - \epsilon, \infty)$, we conclude that 
\begin{itemize}
	\item $\frac{d }{dx}q_{\lambda, z}(x)>0$ for  all $x \in [0,r_1(z))$,  $\frac{d }{dx}q_{\lambda, z}(x)<0$ for $x \in (r_1(z), r_2(z))$, and $\frac{d }{dx}q_{\lambda, z}(x)>0$ for  all $x \in (r_2(z), \infty)$ when $z \in [2\sqrt{\lambda} - \epsilon, \lambda/\epsilon)$; 
	\item $\frac{d }{dx}q_{\lambda, z}(x)<0$ for all $x \in [0,r_2(z))$ and $\frac{d }{dx}q_{\lambda, z}(x)>0$ for all $x \in (r_2(z), \infty)$ when $z\in (\lambda/\epsilon, \infty)$.
\end{itemize} 
This gives statement (ii). 
 \end{proof}

\begin{lemma}\label{lemma:qconvexity}
If $\sqrt{\lambda} \leq \epsilon$, then $q_{\lambda, z}$ is convex on $[0, \infty)$. If $\sqrt{\lambda} > \epsilon$, then $q_{\lambda, z}$  is concave on $[0, \sqrt{\lambda} - \epsilon]$ and convex on $[\sqrt{\lambda} - \epsilon, \infty)$.
\end{lemma}

\begin{proof}
This follows immediately from the fact that $\frac{d^2}{dx^2} q_{\lambda, z}(x) = \frac{1}{\lambda} - \frac{1}{(x+\epsilon)^2}$ for $x \in (0, \infty)$.
 \end{proof}

When $\sqrt{\lambda} \leq {\epsilon}$, Lemmas \ref{lemma:qproperties} and \ref{lemma:qconvexity} imply that $q_{\lambda, z}$ has a unique minimizer for each $z$. In other words, $\prox_{\lambda g}$ will be single-valued. To study $\mathrm{prox}_{\lambda g}$ in the case of  $\sqrt{\lambda} > \epsilon$, we need the following two lemmas.

\begin{lemma}\label{lemma:z1z2}
Let $0\le u<v$. If $\alpha \in \mathrm{prox}_{\lambda g}(u)$ and $\beta \in \mathrm{prox}_{\lambda g}(v)$, then $0 \le \alpha \le \beta$.
\end{lemma}
\begin{proof}\ \ By the definition of the proximity operator, one has $q_{\lambda, u}(\alpha) \le q_{\lambda, u}(\beta)$ and $q_{\lambda, v}(\beta) \le q_{\lambda, v}(\alpha)$. Then, $q_{\lambda, u}(\alpha)+ q_{\lambda, v}(\beta) \le q_{\lambda, u}(\beta)+q_{\lambda, v}(\alpha)$. After simplification, we have from the previous inequality that  $(\alpha-\beta)(u-v)\ge 0$. Hence, $\alpha \le \beta$.
 \end{proof}

\begin{lemma}\label{lemma:singleton}
If the set $\mathrm{prox}_{\lambda g}(z_*)$ at some $z_*>0$ contains zero and a positive number, then $\mathrm{prox}_{\lambda g}(z)$ is a singleton for all $|z|\neq z_*$. In particular,  $\mathrm{prox}_{\lambda g}(z)=\{0\}$ for all $|z|<z_*$ and $\mathrm{prox}_{\lambda g}(z)$ contains only one nonzero element for all $|z|>z_*$.
\end{lemma}
\begin{proof}\ \ By Lemma~\ref{lemma:ProxzSmallerz} and equation~\eqref{eq:prox-g-0}, we only need to consider $\mathrm{prox}_{\lambda g}(z)$ for $z>0$. Since $0 \in \mathrm{prox}_{\lambda g}(z_*)$, then $\mathrm{prox}_{\lambda g}(z)=\{0\}$ for all $0\le z<z_*$ by Lemma~\ref{lemma:z1z2}. Let $\alpha>0$ be an element in $\mathrm{prox}_{\lambda g}(z_*)$. Then, by Lemma~\ref{lemma:z1z2} again, all elements in $\mathrm{prox}_{\lambda g}(z)$ must greater than or equal to $\alpha$ for all $z>z_*$.

Suppose that  $\mathrm{prox}_{\lambda g}(z)$ for some $z >z_*$ has at least two nonzero elements, say $\beta$ and $\gamma$. Then, one should have $q_{\lambda, z}(\beta)=q_{\lambda, z}(\gamma)$ and $\frac{d}{dx}q_{\lambda, z}(\beta)=\frac{d}{dx}q_{\lambda, z}(\gamma)=0$. The  intermediate value theorem implies that there exists another point between $\beta$ and $\gamma$, say $\tau$, at which $\frac{d}{dx}q_{\lambda, z}(\tau)=0$. However,  $\frac{d}{dx}q_{\lambda, z}(x) = \frac{1}{\lambda}(x-z)+\frac{1}{x+\epsilon}$ has at most two roots on $[0, \infty)$. We conclude that $\mathrm{prox}_{\lambda g}(z)$ is a singleton for all $z>z_*$. This completes the proof.
 \end{proof}

Before closing this section,  we present one property of $\mathrm{prox}_{\lambda g}$ that will be used in Section \ref{sec:Low-Rank}. Define $\mathbb{R}^n_\downarrow:=\{x \in \mathbb{R}^n: x_1 \ge x_2 \ge \cdots \ge x_n \ge 0\}$.
\begin{lemma}\label{lemma:f-order}
For any $\lambda>0$ and $\epsilon>0$, if $x \in \mathbb{R}^n_\downarrow$, then $\mathrm{prox}_{\lambda g}(x) \subset \mathbb{R}^n_\downarrow$.
\end{lemma}
\begin{proof}\ \
This is a direct consequence of Lemma~\ref{lemma:z1z2}.
 \end{proof}

\subsection{The Proximity Operators of $g$ and $f$}\label{subsub:1}
We are now prepared to compute the proximity operator of $g$. As above, the problem is split between two cases: $\sqrt{\lambda} \leq \epsilon$ and $\sqrt{\lambda} > \epsilon$, i.e., the convex case and the nonconvex case.  Due to Lemmas~\ref{lemma:ProxzSmallerz} and \ref{lemma:qconvexity},  $\mathrm{prox}_{\lambda g}$ is a single-valued operator when $\sqrt{\lambda}\le \epsilon$. More precisely, we have the following result.
\begin{proposition}\label{prop:prox<epsilon}
If $\sqrt{\lambda}\le \epsilon$, then
\begin{equation}
\mathrm{prox}_{\lambda g}(z) =\left\{
                                \begin{array}{ll}
                                  \{0\}, & \hbox{if $|z| \le \frac{\lambda}{\epsilon};$} \\
                                  \{\mathrm{sgn}(z) r_2(|z|)\}, & \hbox{if $|z| > \frac{\lambda}{\epsilon}$,}
                                \end{array}
                              \right.
\end{equation}
where $r_2$ is given in \eqref{def:r2}.
\end{proposition}
\begin{proof}\ \  Assume that $z \ge0$ in the following discussion.

By Lemma~\ref{lemma:qproperties}, when $z \in [0, \frac{\lambda}{\epsilon}]$, the function $q_{\lambda, z}$ is increasing on $[0, \infty)$, hence  $\mathrm{prox}_{\lambda g}(z)=\{0\}$; when $z \in [\frac{\lambda}{\epsilon}, \infty)$, $q_{\lambda, z}$ is decreasing $[0, r_2(z)]$ and increasing on $[r_2(z), \infty)$,  hence $\mathrm{prox}_{\lambda g}(z)=\{r_2(z)\}$.
 \end{proof}

Recall from Lemma \ref{lemma:singleton} that $\prox_{\lambda g}$ is single-valued except possibly at $\pm z_*$ for some $z_* \in \mathbb{R}$. As we will see in the next result, the point $z_*$ does exist and can be efficiently located when $\sqrt{\lambda}> \epsilon$. In this scenario, $\mathrm{prox}_{\lambda g}$ is described in the following result.
\begin{proposition}\label{prop:prox>epsilon}
If $\sqrt{\lambda}> \epsilon$,  then for any given $z\in \mathbb{R}$
\begin{equation}
\mathrm{prox}_{\lambda g} (z) =
\left\{
  \begin{array}{ll}
    \{0\}, & \hbox{if $|z| < z_*;$}\\
    \{0,\mathrm{sgn}(z) r_2(z_*)\}, & \hbox{if $|z|=z_*;$}\\
    \{\mathrm{sgn}(z) r_2(|z|)\} , & \hbox{if $|z| > z_*$},
  \end{array}
\right.
\end{equation}
$z_*$ is the root of the function
\begin{equation*}\label{def:r}
r(z):=q_{\lambda, z}(r_2(z))-q_{\lambda, z}(0)
\end{equation*}
on the interval $[2\sqrt{\lambda}-\epsilon, \frac{\lambda}{\epsilon}]$ and $r_2(z)$ is given in \eqref{def:r2}.
\end{proposition}
\begin{proof} \ \  As before, we restrict our attention to $z \geq 0$.

By Lemma~\ref{lemma:qproperties}, when $z \in [0,  2\sqrt{\lambda}-\epsilon]$, $q_{\lambda, z}$ is increasing on $[0, \infty)$, hence $\mathrm{prox}_{\lambda g} (z)=\{0\}$; when $z \in [\frac{\lambda}{\epsilon}, \infty)$, $q_{\lambda, z}$ is decreasing on $[0, r_2(z)]$ and increasing on $[r_2(z), \infty)$, hence $\mathrm{prox}_{\lambda g} (z)=\{r_2(z)\}$.

Next, we focus on the situation of $z \in [2\sqrt{\lambda}-\epsilon, \frac{\lambda}{\epsilon}]$. By Lemma~\ref{lemma:qproperties}, $q_{\lambda, z}$ is increasing $[0, r_1(z)]$, decreasing on $[r_1(z),r_2(z)]$, and increasing on $[r_2(z), \infty)$. Hence the elements of $\mathrm{prox}_{\lambda g} (z)$ must be $0$, $r_2(z)$, or both. To determine when the origin and/or $r_2$ are the minimizers of $q_{\lambda, z}$, we take a closer look at the function $r(z)$ in \eqref{def:r} for $2\sqrt{\lambda}-\epsilon < z \le \frac{\lambda}{\epsilon}$. One can check directly that
$$
r(2\sqrt{\lambda}-\epsilon)=-\log\left(\frac{\epsilon}{\sqrt{\lambda}}\right)+\left(\frac{\epsilon^2}{2\lambda}+\frac{2\epsilon}{\sqrt{\lambda}}-\frac{3}{2}\right)>0
$$
and
$$
r\left(\frac{\lambda}{\epsilon}\right)=\frac{\epsilon^2}{2\lambda}+\log\left(\frac{\lambda}{\epsilon^2}\right) -\frac{\lambda}{2\epsilon^2}<0
$$
whenever $\sqrt{\lambda}> \epsilon$. Hence, the function $r$ has at least one root $z_* \in (2\sqrt{\lambda}-\epsilon, \frac{\lambda}{\epsilon})$. This means that $\{0, r_2(z_*)\}=\mathrm{prox}_{\lambda g} (z_*)$. By Lemma~\ref{lemma:singleton}, $z_*$ is the only root of $r$ on the interval $[2\sqrt{\lambda}-\epsilon, \frac{\lambda}{\epsilon}]$. Furthermore, we conclude that $\mathrm{prox}_{\lambda g} (z)=\{0\}$ when $z \in [2\sqrt{\lambda}-\epsilon, z_*)$ and $\mathrm{prox}_{\lambda g} (z)=\{r_2(z)\}$ when $z \in (z_*,\frac{\lambda}{\epsilon}]$.
 \end{proof}

From the above proof, we know that $z_*$ is the unique root of $r$ on $[2\sqrt{\lambda}-\epsilon, \frac{\lambda}{\epsilon}]$ and depends on parameters $\lambda$ and $\epsilon$ only. Furthermore, from $r(2\sqrt{\lambda}-\epsilon) r(\frac{\lambda}{\epsilon})<0$, $z_*$ can easily be found by the bisection method. We further remark that the $z_*$ is simply given as $2\sqrt{\lambda}-\epsilon$ in \cite{Xia-Wang-Meng-Chai-Liang:JCM-2017}, which clearly is incorrect.

Figure~\ref{fig1} displays the proximity operator $\mathrm{prox}_{\lambda g}$ for two choices of $(\lambda,\epsilon)$. Figure~\ref{fig1} (a) depicts $\mathrm{prox}_{\lambda g}$ for $\sqrt{\lambda} \le \epsilon$ with $(\lambda,\epsilon)=(2,3)$ as in Proposition~\ref{prop:prox<epsilon}. Figure~\ref{fig1} (b) depicts $\mathrm{prox}_{\lambda g}$ for $\sqrt{\lambda} > \epsilon$ with $(\lambda,\epsilon)=(3,1)$, corresponding to Proposition~\ref{prop:prox>epsilon}. In both situations, $\mathrm{prox}_{\lambda g}(z)=\{0\}$ for $z$ in a neighborhood of the origin, thus $g$ is a sparsity promoting function as defined in \cite{Shen-Suter-Tripp:JOTA:2019}.
\begin{figure}
\centering
\begin{tabular}{cc}
\includegraphics[width=2.0in]{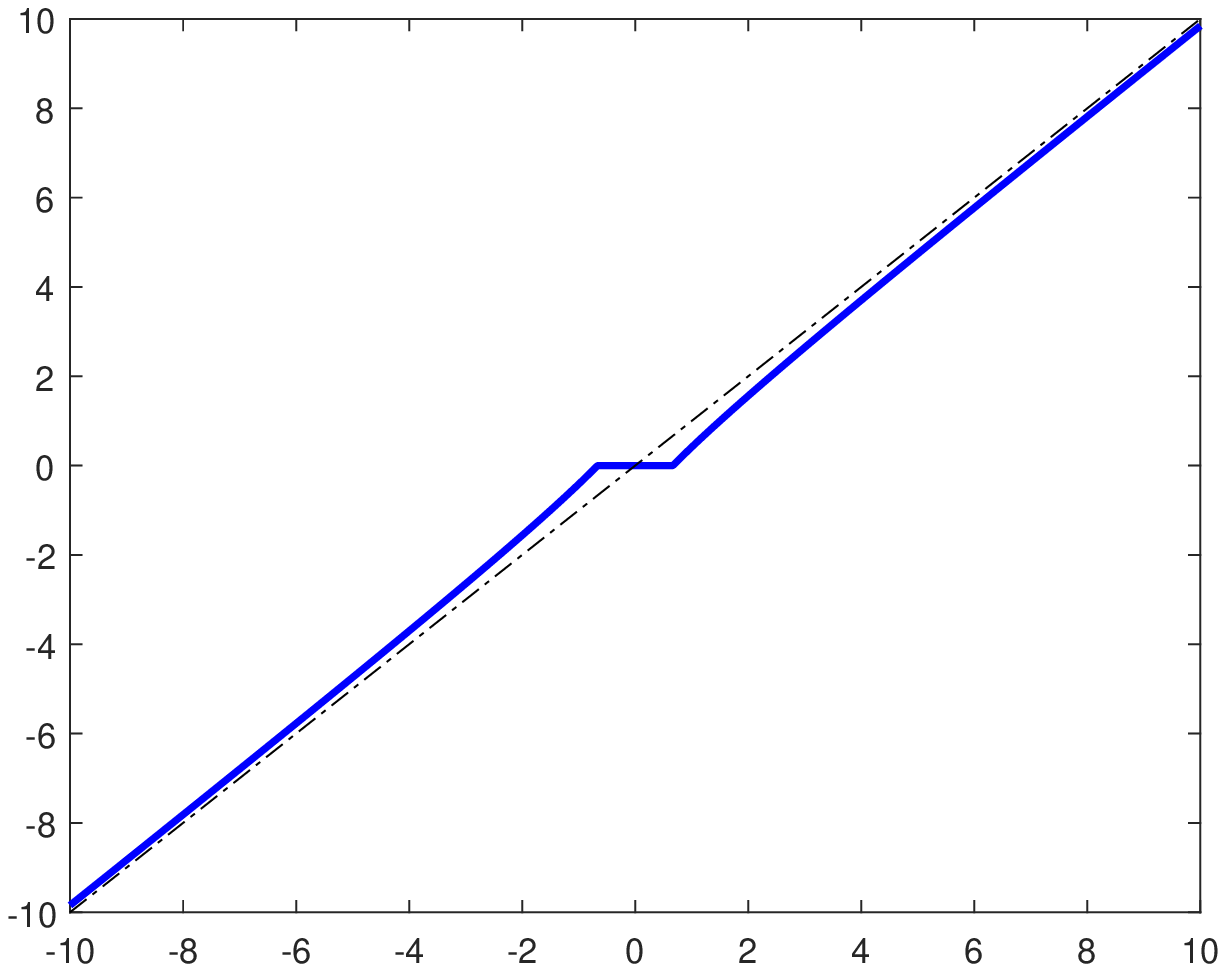}&
\includegraphics[width=2.0in]{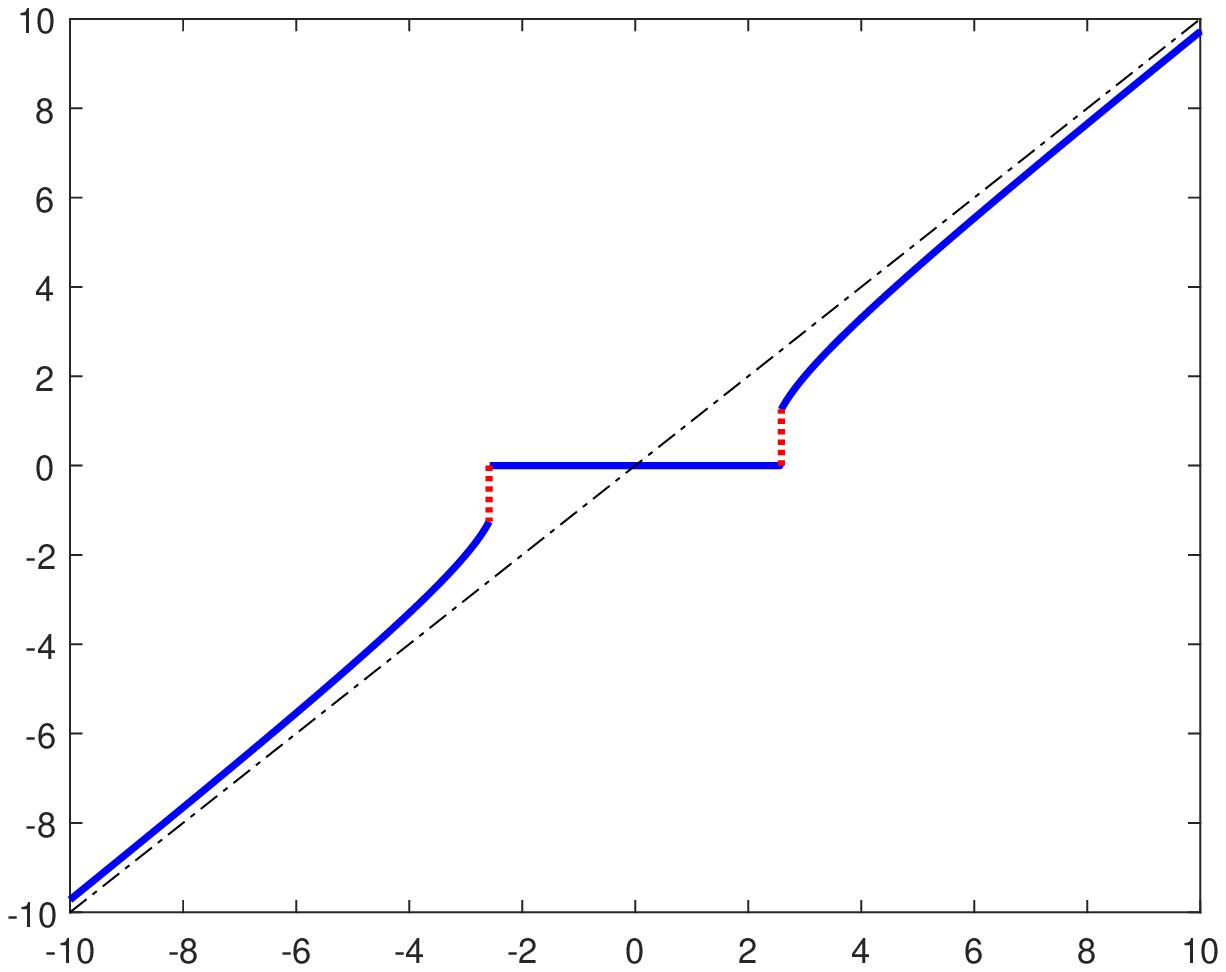}\\
(a)&(b)
\end{tabular}
\caption{The graphs (solid lines) of $\mathrm{prox}_{\lambda g}$ for (a) $\sqrt{\lambda} \le \epsilon$ with {$(\lambda,\epsilon)=(2,3)$} and (b) $\sqrt{\lambda} > \epsilon$ with {$(\lambda,\epsilon)=(3,1)$}. The dotted lines are the graph of the identity mapping.}
\label{fig1}
\end{figure}

Both Propositions~\ref{prop:prox<epsilon} and \ref{prop:prox>epsilon} show that for large enough $|z|$, $r_2(|z|)$ the absolute value of the only element of $\mathrm{prox}_{\lambda g}(z)$ has 
$$
r_2(|z|) \approx |z|-\frac{\lambda}{(|z|+\epsilon)}-\frac{\lambda^2}{(|z|+\epsilon)^3}
$$
through Taylor's expansion the term $\sqrt{1-\frac{4\lambda}{(|z|+\epsilon)^2}}$ in  $\sqrt{\frac{1}{4}(|z|+\epsilon)^2-\lambda}=\frac{1}{2}(|z|+\epsilon)\sqrt{1-\frac{4\lambda}{(|z|+\epsilon)^2}}$.  This indicates that the operator $\mathrm{prox}_{\lambda g}$ is nearly unbiased for large values~\cite{Fan-Li:JASA:01,Shen-Suter-Tripp:JOTA:2019}, which supports the use of $g$ in applications to replace the $\ell_0$ norm. We are not aware of any existing  work quantitatively explaining it in this way. Figure~\ref{fig1} further illustrates this claim.

We are ready to present the solution to problem~\eqref{eq:prox-f-CS}. Recall from \eqref{eq:f-g} that $\prox_{\lambda f}$ acts on separately each coordinate of $z$: $$\prox_{\lambda f}(z) = \prox_{\lambda g}(z_1) \times \cdots \times \prox_{\lambda g}(z_n),$$ where $f(z) = \sum_{i=1}^n g(z_i)$ and $z = (z_1, \dots, z_n)$.

\begin{theorem}\label{thm:prox-f-CS}
For each $z \in \mathbb{R}^n$, $\epsilon>0$, and $\lambda \le 0$, the set $\mathrm{prox}_{\lambda f}(z)$ collects all minimizers to problem~\eqref{eq:prox-f-CS}. Moreover, if $x^\star \in \mathrm{prox}_{\lambda f}(z)$, then $x^\star_i \in \mathrm{prox}_{\lambda g}(z_i)$, $i=1,2,\ldots,n$.
\end{theorem}
\begin{proof} \ \
The results follow immediately from the relation \eqref{eq:f-g} and the definition of proximity operator. Since  $\mathrm{prox}_{\lambda g}$ has explicit expression, so does $\mathrm{prox}_{\lambda f}$.
 \end{proof}

\section{The Reweighted $\ell_1$ Algorithm May Fail for $\bm{ \mathrm{prox}_{\lambda g}}$}\label{sec:rwl1}
In the existing work, such as \cite{Candes-Wakin-Boyd:JFAA:08,Dong-Shi-Li-Ma-Huang:IEEEIP:2014}, the iteratively reweighted $\ell_1$ method is adopted to evaluate $\mathrm{prox}_{\lambda g}$. This iterative algorithm is in the fashion of classical majorization-minimization algorithms: it generates and solves a sequence of convex optimization problems. At each iteration, the non-convex function $g$ approximated by means of a majorizing convex surrogate function. More precisely, assume that $\x{k}$ is the value of the current solution. The chosen convex surrogate function of $g$ is $g(\x{k})+\frac{|x|-|\x{k}|}{\epsilon+|\x{k}|}$ the first-order approximation of $g(x)$ at $\x{k}$. With it, the next iteration $x^{(k+1)}$ is
$$
x^{(k+1)} = \argmin\left\{\frac{1}{2\lambda} (x-z)^2 + g(\x{k}) + \frac{|x|-|\x{k}|}{\epsilon+|\x{k}|}: x \in \mathbb{R}\right\},
$$
which can be viewed as a reweighted $\ell_1$-minimization problem. This algorithm can be written as
\begin{equation}\label{eq:iterative}
x^{(k+1)} = \mathrm{prox}_{ \frac{\lambda}{\epsilon+|\x{k}|}|\cdot|}(z)
=\left\{
\begin{array}{ll}
0,& \hbox{$|z| \le \frac{\lambda}{\epsilon+|\x{k}|}$;}\\[8 pt]
\mathrm{sgn}(z) (|z|-\frac{\lambda}{\epsilon+|\x{k}|}), & \hbox{$|z| > \frac{\lambda}{\epsilon+|\x{k}|}$.} \\
\end{array}
\right.
\end{equation}

We will see that the sequence $\{\x{k}\}$  generated by the iterative scheme \eqref{eq:iterative} is always convergent and its limit depends on the initial guess $\x{0}$ and the relation of $z$ with the parameters $\lambda$ and $\epsilon$. Lemmas~\ref{lemma:conv1} - \ref{lemma:conv6} in the following describe the possible convergence behavior of \eqref{eq:iterative}. This is then compared to the true solution in Theorems~\ref{thm:alg-accuracy-1} and \ref{thm:alg-accuracy}. In particular, we identify the intervals where \eqref{eq:iterative} will not achieve the true solution. These intervals are explicitly determined in terms of the initial guess $\x{0}$ and parameters $\lambda$ and $\epsilon$.

\begin{lemma}\label{lemma:conv1}
Let the sequence $\{\x{k}\}$ be generated by the iterative scheme \eqref{eq:iterative} for a given $z>0$ and an initial guess $\x{0} \ge 0$. Suppose that  $\x{k} > 0$ for all $k \ge 1$. Then the sequence $\{\x{k}\}$ is increasing (resp. decreasing) if $\x{1}>\x{0}$ (resp. $\x{1}<\x{0}$); this sequence is constant if $\x{1}=\x{0}$.
\end{lemma}
\begin{proof} Since $\x{k} > 0$ for all $k \ge 1$, one has from the iterative scheme \eqref{eq:iterative} that $x^{(k+1)}=z-\frac{\lambda}{\epsilon+\x{k}}$ for all $k \ge 0$. Therefore,
$$
x^{(k+1)}-\x{k}= \frac{\lambda (\x{k}-x^{(k-1)})}{(\epsilon+\x{k})(\epsilon+x^{(k-1)})}=\frac{\lambda^k (\x{1}-\x{0})}{\Pi_{i=1}^k ((\epsilon+x^{(i)})(\epsilon+x^{(i-1)}))}.
$$
All statements immediately follow from the above equation.
 \end{proof}

\begin{lemma}\label{lemma:conv2}
Let the sequence $\{\x{k}\}$ be generated by the iterative scheme \eqref{eq:iterative} for a given $z\ge 0$ and an initial guess $\x{0} \ge 0$. If there exists $k_0 \ge 0$, such that $x^{(k_0)}=0$, then the sequence $\{\x{k}\}$ converges to $0$ if $z \le \frac{\lambda}{\epsilon}$ and to $r_2(z)$ if $z>\frac{\lambda}{\epsilon}$.
\end{lemma}
\begin{proof}\ \ Without loss of generality, let us assume $k_0=0$, i.e, $\x{0}=0$. We have
$$
\x{1}=\left\{
          \begin{array}{ll}
            0, & \hbox{if $z\le \frac{\lambda}{\epsilon}$;} \\
            z-\frac{\lambda}{\epsilon}, & \hbox{if $z> \frac{\lambda}{\epsilon}$.}
          \end{array}
        \right.
$$
Obviously, if $z\le \frac{\lambda}{\epsilon}$, $\x{k}=0$ for all $k\ge 0$, that is,  $\{\x{k}\}$ converges to $0$. If $z> \frac{\lambda}{\epsilon}$, then $\x{1}>\x{0}=0$, yielding $x^{(k+1)}=z-\frac{\lambda}{\epsilon+\x{k}}>0$ for all $k\ge 0$.  So $\{\x{k}\}$ is increasing by Lemma~\ref{lemma:conv1} and converges, say to a positive number $x^{\infty}$, due to $0\le \x{k} < z$ for all $k\ge 0$. We have
$x^{\infty}=z-\frac{\lambda}{\epsilon+x^{\infty}}$. So $x^{\infty}$ must be $r_2(z)$ for $z>\frac{\lambda}{\epsilon}$.
 \end{proof}

The next identity is useful in the following discussion. For given $\x{0} \ge 0$ and $z >0$, if $\x{1}=z-\frac{\lambda}{\epsilon+\x{0}}>0$, then
\begin{equation}\label{eq:important}
\x{1}-\x{0}=\left\{
                  \begin{array}{ll}
                    -\frac{1}{\epsilon+\x{0}}((\x{0}-\frac{1}{2}(z-\epsilon))^2+ (\lambda-\frac{1}{4}(z+\epsilon)^2)), & \hbox{if $z < 2\sqrt{\lambda}-\epsilon$;} \\
                    -\frac{1}{\epsilon+\x{0}}(\x{0}-r_1(z))(\x{0}-r_2(z)), & \hbox{if $z \ge 2\sqrt{\lambda}-\epsilon$,}
                  \end{array}
                \right.
\end{equation}
where $r_1(z)$ and $r_2(z)$ are given in \eqref{def:r1} and \eqref{def:r2}, respectively.

\begin{lemma}\label{lemma:conv3}
Let the sequence $\{\x{k}\}$ be generated by the iterative scheme \eqref{eq:iterative} for a given $z\ge \frac{\lambda}{\epsilon}$ and an initial guess $\x{0} > 0$. Then, $\x{k} >0$ for all $k \ge 0$ and the sequence $\{\x{k}\}$ converges to $r_2(z)$.
\end{lemma}
\begin{proof} \ \ For $z\ge \frac{\lambda}{\epsilon}$ and $\x{0} > 0$, we know $\x{1}=z-\frac{\lambda}{\epsilon+\x{0}}>0$. As a consequence, it also implies $\x{k} >0$ for all $k \ge 0$. To show the convergence of the sequence $\{\x{k}\}$, we compare the values of $\x{0}$ and $\x{1}$ from \eqref{eq:important} in order to infer the monotonicity of the sequence based on Lemma~\ref{lemma:conv1}. To this end, our discussion is conducted for two cases: (i) $z> \frac{\lambda}{\epsilon}$ or $z= \frac{\lambda}{\epsilon}$ and $\sqrt{\lambda}>\epsilon$; and (ii) $z= \frac{\lambda}{\epsilon}$ and $\sqrt{\lambda} \le \epsilon$. The following facts are useful: $r_1(z) < 0 < r_2(z)$ for all $z> \frac{\lambda}{\epsilon}$, and
$$
r_1\left(\frac{\lambda}{\epsilon}\right)=\left\{
                                \begin{array}{ll}
                                  \frac{\lambda}{\epsilon}-\epsilon, & \hbox{if $\sqrt{\lambda} \le \epsilon$;} \\
                                  0, & \hbox{if $\sqrt{\lambda} > \epsilon$}
                                \end{array}
                              \right. \quad \mbox{and} \quad
r_2\left(\frac{\lambda}{\epsilon}\right)=\left\{
                                \begin{array}{ll}
                                  0, & \hbox{if $\sqrt{\lambda} \le \epsilon$;} \\
                                  \frac{\lambda}{\epsilon}-\epsilon, & \hbox{if $\sqrt{\lambda} > \epsilon$.}
                                \end{array}
                              \right.
$$

Case (i): $z> \frac{\lambda}{\epsilon}$ or $z= \frac{\lambda}{\epsilon}$ and $\sqrt{\lambda}>\epsilon$.  Hence, $r_2(z)>0$ and $r_2(z)$ is the only positive solution of $x=z-\frac{\lambda}{\epsilon+x}$.  
\begin{itemize}
\item If $\x{0} \in (0, r_2(z))$, one has  $\x{1} > \x{0}$ from~\eqref{eq:important}. We can conclude that $x^{(k+1)} > \x{k}$ for all $k \ge 0$ and the sequence $\{\x{k}\}$ converges, say to $x^{\infty}$, which is a positive number satisfying $x^{\infty}=z-\frac{\lambda}{\epsilon+x^{\infty}}$. So $x^{\infty}=r_2(z)$.

\item If $\x{0}=r_2(z)>0$, then  $x^{(k+1)}=r_2(z)$ for all $k \ge 0$. Hence, the limit of the sequence $\{\x{k}\}$ is $r_2(z)$.

\item If $\x{0} \in (r_2(z), \infty)$, then  $\x{1} < \x{0}$ from~\eqref{eq:important}. We conclude  $x^{(k+1)} < \x{k}$ for all $k \ge 0$ and the sequence $\{\x{k}\}$ converges, say to $x^{\infty} \ge 0$, satisfying  $x^{\infty}=z-\frac{\lambda}{\epsilon+x^{\infty}}$. Hence, $x^{\infty}$ must be $r_2(z)$.

\end{itemize}

Case (ii): $z= \frac{\lambda}{\epsilon}$ and $\sqrt{\lambda} \le \epsilon$. If $\sqrt{\lambda} \le \epsilon$, then $r_2(z)=0$, so $\x{1} < \x{0}$ from~\eqref{eq:important}. Then $x^{(k+1)} < \x{k}$ for all $k \ge 0$ and the sequence $\{\x{k}\}$ converges, say to $x^{\infty} \ge 0$, satisfying  $x^{\infty}=z-\frac{\lambda}{\epsilon+x^{\infty}}$. Hence, $x^{\infty}$ must be $r_2(z)=0$.

From the discussion above, we know that the sequence $\{\x{k}\}$ converges to $r_2(z)$.
 \end{proof}

It can be concluded from Lemma~\ref{lemma:conv2} and Lemma~\ref{lemma:conv3} that the sequence $\{\x{k}\}$ always converges to $r_2(z)$ for $z>\frac{\lambda}{\epsilon}$ regardless of the initial guess $\x{0}$. The next lemma shows that the sequence $\{\x{k}\}$ always converges to $0$ for all $z \in (0, 2\sqrt{\lambda}-\epsilon)$ if $2\sqrt{\lambda}-\epsilon >0$, independent of the initial guess $\x{0}$.

\begin{lemma}\label{lemma:conv4}
Suppose $2\sqrt{\lambda}-\epsilon >0$. Let the sequence $\{\x{k}\}$ be generated by the iterative scheme \eqref{eq:iterative} for a given $z \in (0, 2\sqrt{\lambda}-\epsilon)$ and an initial guess $\x{0} \ge 0$. Then the sequence $\{\x{k}\}$ converges to $0$.
\end{lemma}
\begin{proof}\ \ Notice that $2\sqrt{\lambda}-\epsilon \le \frac{\lambda}{\epsilon}$ for all positive $\lambda$ and $\epsilon$. If there exists $k_0 \ge 0$ such that $x^{(k_0)}=0$, by Lemma~\ref{lemma:conv2}, the sequence $\{\x{k}\}$ converges to $0$.

Now, assume that all elements $\x{k}$ are positive. By \eqref{eq:important}, we have $\x{1}<\x{0}$ for all $z \in (0, 2\sqrt{\lambda}-\epsilon)$ and an initial guess $\x{0} > 0$. By Lemma~\ref{lemma:conv2}, the sequence $\{\x{k}\}$ is decreasing and convergent. Suppose that $\lim_{k\rightarrow \infty} \x{k}=x^{\infty} \ge 0$. Then, $x^{\infty}=z-\frac{\lambda}{\epsilon+x^{\infty}}$ due to all $\x{k}>0$, but it is impossible for $z \in (0, 2\sqrt{\lambda}-\epsilon)$. We conclude that the sequence $\{\x{k}\}$ converges to $0$.
 \end{proof}

The next two lemmas deal with the convergence of the sequence $\{\x{k}\}$ for $z \in [2\sqrt{\lambda}-\epsilon, \frac{\lambda}{\epsilon})$.

\begin{lemma}\label{lemma:conv5}
Suppose parameters $\lambda$ and $\epsilon$ satisfying conditions $2\sqrt{\lambda}-\epsilon >0$ and $\sqrt{\lambda} \le \epsilon$. Let the sequence $\{\x{k}\}$ be generated by the iterative scheme \eqref{eq:iterative} for a given $z \in [2\sqrt{\lambda}-\epsilon, \frac{\lambda}{\epsilon})$ and an initial guess $\x{0} \ge 0$. Then the sequence $\{\x{k}\}$ converges to $0$.
\end{lemma}
\begin{proof}\ \ If there exists $k_0 \ge 0$ such that $x^{(k_0)}=0$, by Lemma~\ref{lemma:conv2}, the sequence $\{\x{k}\}$ converges to $0$.

Now, assume that all elements $\x{k}$ are positive. Since $r_1(z) < 0$ and $r_2(z) < 0$  for $z \in [2\sqrt{\lambda}-\epsilon, \frac{\lambda}{\epsilon})$, we know from \eqref{eq:important} that $\x{1} < \x{0}$ if $\x{0} \in (0,\infty)$.  Then, the sequence $\{\x{k}\}$ is decreasing by Lemma~\ref{lemma:conv1}, and therefore convergent. Suppose that $\lim_{k\rightarrow \infty} \x{k}=x^{\infty} \ge 0$. One has $x^{\infty}=z-\frac{\lambda}{\epsilon+x^{\infty}}$. Therefore, $x^{\infty}$ should be $r_1(z)$ or $r_2(z)$. However, it is impossible due to both $r_1(z)$ and $r_2(z)$ are negative for $z \in [2\sqrt{\lambda}-\epsilon, \frac{\lambda}{\epsilon})$.
 \end{proof}

\begin{lemma}\label{lemma:conv6}
Suppose $\sqrt{\lambda}> \epsilon$. Let the sequence $\{\x{k}\}$ be generated by the iterative scheme \eqref{eq:iterative} for a given $z \in [2\sqrt{\lambda}-\epsilon, \frac{\lambda}{\epsilon})$ and an initial guess $\x{0} \ge 0$. Then the following statements hold
\begin{enumerate}[(i)]
\item If $\x{0} \in (0, r_1(z))$, then the sequence $\{\x{k}\}$ converges to $0$; If $\x{0}=r_1(z)$, then the sequence $\{\x{k}\}$ converges to $r_1(z)$.

\item If $\x{0} \in (r_1(z),\infty)$, then  the sequence $\{\x{k}\}$ converges to $r_2(z)$.

\end{enumerate}
\end{lemma}
\begin{proof}\ \ First, we show that
\begin{equation}\label{eq:<r1z}
\frac{\lambda}{z}-\epsilon < r_1(z)
\end{equation}
holds for all $z \in [2\sqrt{\lambda}-\epsilon, \frac{\lambda}{\epsilon})$. Actually, by the definition of $r_1(z)$ in \eqref{def:r1} and through some manipulations, inequality~\eqref{eq:<r1z} is equivalent to $\sqrt{(z+\epsilon)^2-4\lambda} < (z+\epsilon)-\frac{2\lambda}{z}$. Since the expression  $(z+\epsilon)-\frac{2\lambda}{z}$ is positive for $z \in [2\sqrt{\lambda}-\epsilon, \frac{\lambda}{\epsilon})$, squaring the previous inequality followed by some simplifications yields $z < \frac{\lambda}{\epsilon}$, which is obviously true.

(i) If  $\x{0} \le \frac{\lambda}{z}-\epsilon$, then $\x{1}=0$ from \eqref{eq:iterative}. By Lemma~\ref{lemma:conv2}, the sequence $\{\x{k}\}$ converges to $0$.  If  $\x{0} > \frac{\lambda}{z}-\epsilon$, then $0<\x{1} \le \x{0}$ from \eqref{eq:important}. Using a similar argument above, if there exists $k_0 \ge 0$ such that $x^{(k_0))}=0$, by Lemma~\ref{lemma:conv2}, the sequence $\{\x{k}\}$ converges to $0$. Now, assume that all elements $\x{k}$ are positive. Then the sequence $\{\x{k}\}$ is decreasing and convergent  by Lemma~\ref{lemma:conv1}. Suppose that $\lim_{k\rightarrow \infty} \x{k}=x^{\infty} \ge 0$. Then, $x^{\infty}$ must be strictly less than $r_1(z)$ and $x^{\infty}=z-\frac{\lambda}{\epsilon+x^{\infty}}$, which, however, contradict to each other.
Hence, the sequence $\{\x{k}\}$ converges to $0$.

If $\x{0}=r_1(z)$, so $\x{0}>\frac{\lambda}{z}-\epsilon$ which implies that $\x{1}=\x{0}$. In this case, $\{\x{k}\}$ is a constant sequence and its limit is  $r_1(z)$.

(ii) If $\x{0} \in (r_1(z),\infty)$, we have $\x{0} > \frac{\lambda}{z}-\epsilon$ by \eqref{eq:<r1z}. From \eqref{eq:important}, we have  $\x{1} \ge \x{0}$ if $\x{0} \in (r_1(z), r_2(z)]$. So the sequence $\{\x{k}\}$ is increasing and must converge to $r_2(z)$.

From \eqref{eq:important}, we have $\x{1} < \x{0}$ if $\x{0} \in (r_2(z), \infty)$. Further, we can show that $\x{1} > r_2(z)$. Indeed, from $\x{1}-r_2(z)=z-\frac{\lambda}{\epsilon+\x{0}}-r_2(z)$ and the definition of $r_2(z)$, we have after some simplification
$$
\x{1}-r_2(z)=\frac{2(\x{0}-r_2(z))}{(z+\epsilon)+\sqrt{(z+\epsilon)^2-4\lambda}}>0
$$
holds for $z \in [2\sqrt{\lambda}-\epsilon, \frac{\lambda}{\epsilon})$ and  $\x{0} \ge r_2(z)$.  Hence, the sequence $\{\x{k}\}$ is decreasing and $\lim_{k \rightarrow \infty}\x{k}= x^{\infty} \ge r_2(z)$. We further have  $x^{\infty}=z-\frac{\lambda}{\epsilon+x^{\infty}}$ which implies $x^{\infty}=r_2(z)$.
 \end{proof}

Finally, the following two theorems summarize our main results.
\begin{theorem}\label{thm:alg-accuracy-1}
For $\sqrt{\lambda}\le \epsilon$, the iteratively reweighted algorithm can always provide the accurate solution to $\mathrm{prox}_{\lambda g}(z)$ for all $z \in \mathbb{R}$.
\end{theorem}
\begin{proof}\ \  The result  follows directly from Proposition~\ref{prop:prox<epsilon} and Lemmas~\ref{lemma:conv2}-\ref{lemma:conv5}.
 \end{proof}

When $\sqrt{\lambda}> \epsilon$, $r_1$ is a bijection function from $[2\sqrt{\lambda}-\epsilon, \infty)$ to $(-\infty, \sqrt{\lambda}-\epsilon]$.  Therefore, the function $r_1^{-1}$ the inverse of $r_1$ exists and maps $(-\infty, \sqrt{\lambda}-\epsilon]$ to $[2\sqrt{\lambda}-\epsilon, \infty)$.
\begin{theorem}\label{thm:alg-accuracy}
For $\sqrt{\lambda}> \epsilon$ and an initial guess $\x{0}$, the following statements hold for the iteratively reweighted algorithm.
\begin{enumerate}[(i)]
\item If $\x{0} \ge  \sqrt{\lambda}-\epsilon$, the iteratively reweighted algorithm provides the accurate solution to $\mathrm{prox}_{\lambda g}(z)$ for all $z$ on $\mathbb{R}\setminus  ((-z_*, -2\sqrt{\lambda}+\epsilon] \cup [2\sqrt{\lambda}-\epsilon, z_*))$.

\item If $\sqrt{\lambda}-\epsilon > \x{0} > r_1(z_*)$, the iteratively reweighted  algorithm provides the accurate solution to $\mathrm{prox}_{\lambda g}(z)$ for all $z$ on $\mathbb{R}\setminus  ((-z_*, -r_1^{-1}(\x{0})] \cup [r_1^{-1}(\x{0}), z_*))$.

\item If $\x{0}=r_1(z_*)$,  the iterative reweighted algorithm provides the accurate solution to $\mathrm{prox}_{\lambda g}(z)$ for all $z$ on $\mathbb{R}\setminus \{-z_*, z_*\}$.

\item If $r_1(z_*) > \x{0} \ge 0$, the iteratively reweighted algorithm provides the accurate solution to $\mathrm{prox}_{\lambda g}(z)$ for all $z$ on $\mathbb{R}\setminus  ([-r_1^{-1}(\x{0}), -z_*) \cup (z_*, r_1^{-1}(\x{0})]$.
\end{enumerate}
\end{theorem}
\begin{proof}\ \   By Proposition~\ref{prop:prox<epsilon}, Lemma~\ref{lemma:conv2}, and Lemma~\ref{lemma:conv4}, the iteratively reweighted algorithm provides the accurate solution to $\mathrm{prox}_{\lambda g}(z)$ when $|z| > \frac{\lambda}{\epsilon}$ or $|z| < 2\sqrt{\lambda}-\epsilon$. The rest of the proof will focus on the situation for $z\in [2\sqrt{\lambda}-\epsilon), \frac{\lambda}{\epsilon}]$ due to Lemma~\ref{prop:prox<epsilon}.

(i) We know that $\sqrt{\lambda}-\epsilon=r_1(2\sqrt{\lambda}-\epsilon)>r_1(z)$ for all $z\in (2\sqrt{\lambda}-\epsilon, \frac{\lambda}{\epsilon}]$. From Lemma~\ref{lemma:conv3} and Lemma~\ref{lemma:conv6},  the limit of the sequence generated by the algorithm is $r_2(z)$ for  $z\in [2\sqrt{\lambda}-\epsilon), \frac{\lambda}{\epsilon}]$. Hence, the limit does not match to the true solution $\mathrm{prox}_{\lambda g}(z)$ when $z \in (-z_*, -2\sqrt{\lambda}+\epsilon] \cup [2\sqrt{\lambda}-\epsilon, z_*)$.

(ii) Notice that $r_1(z)$ is strictly decreasing on $[2\sqrt{\lambda}-\epsilon), \frac{\lambda}{\epsilon}]$. From $\sqrt{\lambda}-\epsilon > \x{0} > r_1(z_*)$, we have  $2\sqrt{\lambda}-\epsilon <r_1^{-1}(\x{0}) < z_*$; $\x{0}<r_1(z)$ for all $z\in [2\sqrt{\lambda}-\epsilon, r_1^{-1}(\x{0}))$; and
$\x{0}>r_1(z)$ for all $z\in (r_1^{-1}(\x{0}), \frac{\lambda}{\epsilon}]$. Accordingly, by Lemma~\ref{lemma:conv6},  the limit $x^\infty$ of the sequence generated by the algorithm is
\begin{equation}\label{eq:thm3-1}
x^\infty=\left\{
           \begin{array}{ll}
             0, & \hbox{if $z\in [2\sqrt{\lambda}-\epsilon, r_1^{-1}(\x{0}))$;} \\
             \x{0}, & \hbox{if $z=r_1^{-1}(\x{0})$;}\\
             r_2(z), & \hbox{if $z\in (r_1^{-1}(\x{0}), \frac{\lambda}{\epsilon}]$.}
           \end{array}
         \right.
\end{equation}
Hence, the limit does not match to the true solution $\mathrm{prox}_{\lambda g}(z)$ when $z \in (-z_*, -r_1^{-1}(\x{0})] \cup [r_1^{-1}(\x{0}), z_*)$.

(iii) It is directly from Lemma~\ref{lemma:conv6} and the fact of $0<r_1(z_*)<r_2(z_*)$.

(iv) Using similar arguments in item (ii), from $r_1(z_*) > \x{0} \ge 0$, we have  $z_* <r_1^{-1}(\x{0}) \le \frac{\lambda}{\epsilon}$; $\x{0}<r_1(z)$ for all $z\in [2\sqrt{\lambda}-\epsilon, r_1^{-1}(\x{0}))$; and
$\x{0}>r_1(z)$ for all $z\in (r_1^{-1}(\x{0}), \frac{\lambda}{\epsilon}]$. Accordingly, by Lemma~\ref{lemma:conv6}, for $r_1(z_*) > \x{0} > 0$ the limit $x^\infty$ of the sequence generated by the algorithm is given in \eqref{eq:thm3-1}. Hence, the limit does not match to the true solution $\mathrm{prox}_{\lambda g}(z)$ when $z \in [-r_1^{-1}(\x{0}), z_*) \cup (z_*, r_1^{-1}(\x{0})]$. This statement is also true for $\x{0}=0$ by Lemma~\ref{lemma:conv1}.
 \end{proof}
The results in Theorem~\ref{thm:alg-accuracy} can be visualized through Figure~\ref{fig4}. For each initialization $\x{0}$, the intervals for which $\prox_{\lambda g}(z)$ disagrees with the reweighted $\ell_1$ solution are represented by the solid horizontal lines.

\begin{figure}
\centering
\begin{tabular}{c}
\includegraphics[width=2in]{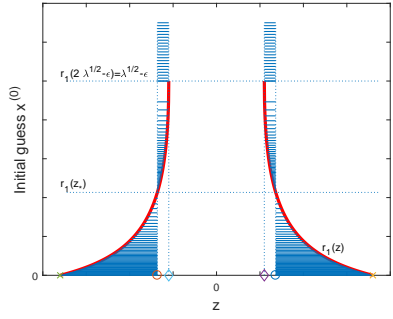}
\end{tabular}
\caption{The intervals on which the iteratively reweighted algorithm will fail.  Here, '$\times$', '$\circ$', and '$\diamond$' denote $\pm\frac{\lambda}{\epsilon}$, $\pm z_*$, and $\pm (2\sqrt{\lambda}-\epsilon)$, respectively.}
\label{fig4}
\end{figure}

\section{Solutions to Optimization Problem~\eqref{eq:prox-f-Low-Rank}}\label{sec:Low-Rank}
In this section, we present the solution to optimization Problem~\eqref{eq:prox-f-Low-Rank}.

Let $\mathcal{M}_{m,n}$ denote the Euclidean space of $m \times n$ real matrices, with inner product ${\langle X, Y \rangle= \mathrm{tr} X^\top Y}$. For any matrix $X \in \mathcal{M}_{m,n}$, let $X_{ij}$ denote its $(i,j)$-th entry. For any vector $x \in \mathbb{R}^{m \land n}$, let $\mathrm{Diag} (x)$
denote the matrix with $(\mathrm{Diag} (x))_{ii}=x_i$ for all $i$, and $(\mathrm{Diag} (x))_{ij}=0$ for $i \neq j$. We want to turn readers attention to the fact that $\mathrm{Diag} (x)$ will denote an $m \times n$ matrix.

For any $X \in \mathcal{M}_{m,n}$, we define ${\sigma(X):=(\sigma_1(X), \sigma_2(X),\ldots, \sigma_{m \land n}(X))^\top}$, where $\sigma_1(X)\ge \sigma_2(X) \ge \ldots \ge \sigma_{m \land n}(X)$ are the ordered singular values of $X$. Denote by $\mathcal{O}(X)$ the set of all pairs $(U,V)$:
$$
\mathcal{O}(X) :=\left\{(U, V) \in \mathcal{M}_{m,m} \times \mathcal{M}_{n,n} : U^\top U=I, V^\top V=I, X=U \mathrm{Diag}(\sigma(X)) V^\top \right\}.
$$
That is, for any pair $(U, V) \in \mathcal{O}(X)$, $U \mathrm{Diag}(\sigma(X)) V^\top$ is a singular value decomposition of $X$.

\begin{theorem}\label{thm:prox-f-Low-Rank}
For each $Z \in \mathcal{M}_{m,n}$, $\epsilon>0$, and $\lambda > 0$, the set $\mathrm{prox}_{\lambda f\circ \sigma}(Z)$ collects all minimizers to problem~\eqref{eq:prox-f-Low-Rank}. Moreover, if $X^\star \in \mathrm{prox}_{\lambda f\circ \sigma}(Z)$, then there exist a pair $(U,V)\in \mathcal{O}(Z)$ and a vector $d \in \mathrm{prox}_{\lambda f}(\sigma(Z))$ such that
\begin{equation}\label{eq:Low-Rank}
X^\star =  U \mathrm{Diag}(d) V^\top.
\end{equation}
\end{theorem}
\begin{proof} \ \
The statement that the set $\mathrm{prox}_{\lambda f\circ \sigma}(Z)$ collects all minimizers to problem~\eqref{eq:prox-f-Low-Rank} is from the definition of proximity operator. Next, we show that $X^\star$ in \eqref{eq:Low-Rank} indeed is a solution to problem~\eqref{eq:prox-f-Low-Rank}.

Problem~\eqref{eq:prox-f-Low-Rank} can be equivalently reformulated as
$$
\min_{d\in \mathbb{R}^{m\land n}_\downarrow}\left\{ \min_{X \in \mathcal{M}_{m,n}, \mathrm{Diag}(\sigma(X))=d}\left\{\frac{1}{2\lambda} \|X-Z\|^2_F+\sum_{i=1}^{m\land n} \log\left(1+\frac{d_i}{\epsilon}\right)\right\}\right\}.
$$
Note that
\begin{eqnarray*}
\|X-Z\|^2_F &=& \mathrm{tr} (X^\top X)-2 \mathrm{tr} (X^\top Z) + \mathrm{tr} (Z^\top Z) \\
&=&\sum_{i=1}^{m \land n} d_i^2 - 2 \mathrm{tr} (X^\top Z) + \sum_{i=1}^{m \land n} \sigma_i(Z)^2 \\
&\ge& \sum_{i=1}^{m \land n} d_i^2 - 2 \sigma(Z)^\top d  + \sum_{i=1}^{m \land n} \sigma_i(Z)^2.
\end{eqnarray*}
The last inequality is due to von Neumann's trace inequality (see \cite{Mirsky:MM:1975}).  Equality holds when $X$ admits the singular value decomposition $X=U \mathrm{Diag}(d) V^\top$, where $(U,V)\in \mathcal{O}(Z)$. Then the optimization problem reduces to
$$
\min_{d\in \mathbb{R}^{m\land n}_\downarrow}\left\{ \sum_{i=1}^{m\land n} \left( \frac{1}{2\lambda} (d_i-\sigma_i(Z))^2 + \log\left(1+\frac{d_i}{\epsilon}\right)\right)\right\}.
$$
The objective function is completely separable and is minimized only when $d_i \in \mathrm{prox}_{\lambda g}(\sigma_i(Z))$. This is a feasible solution because $\sigma(Z) \in  \mathbb{R}^{m\land n}_\downarrow$ implies $\mathrm{prox}_{\lambda f}(\sigma(Z)) \subset \mathbb{R}^{m\land n}_\downarrow$ by Lemma~\ref{lemma:f-order}. This completes the proof.
 \end{proof}
We remark that the main ideas in the above proof are from \cite{Chen-Dong-Chan:Biometrika-2013}. The result in Theorem~\ref{thm:prox-f-Low-Rank} can be applied, for example in \cite{cai2020image,Deng-Dai-Liu-Zhang-Hu:IEEENNLS:13,Dong-Shi-Li-Ma-Huang:IEEEIP:2014,Fazel-Hindi-Boyd:03}, to avoid an inner loop for evaluating $\mathrm{prox}_{\lambda f\circ \sigma}$.

\section{Conclusions}\label{sec:conclusions}
We presented the explicit expressions of the proximity operators of the log-sum penalty and its composition with the singular value function. In the existing work, these proximity operators were computed through the iteratively reweighted $\ell_1$ methods that are inefficient, and may sometimes give inaccurate results, as analyzed in Theorem~\ref{thm:alg-accuracy} and demonstrated in Figure~\ref{fig4}. By applying the results from this paper, one can avoid using inefficient iterative approaches to compute the proximity operator of the log-sum penalty, and can prevent inaccurate solutions from sub-optimal initial values.  Moreover, we have characterized the behavior of the proximity operator for the log-sum penalty, and further justified its use as a nonconvex surrogate in $\ell_0$ and $\ell_1$ norm minimization problems.

\section*{Disclaimer and Acknowledgment of Support}
The work of L. Shen was supported in part by the National Science Foundation under grant DMS-1913039, 2020
U.S.\ Air Force Summer Faculty Fellowship Program, and the 2020 Air Force Visiting Faculty Research Program funded through AFOSR grant 18RICOR029. Any opinions, findings and conclusions or recommendations expressed in this material are those of the authors and do not necessarily reflect the views of the U.S.\ Air Force Research
Laboratory.  Cleared for public release 08 Jan 2021: Case number AFRL-2021-0024. 

\bibliographystyle{siam}

\end{document}